\documentclass[11pt,oneside,british]{article}
\usepackage{jeffplus_arxiv,tgpagella}

\numberwithin{equation}{section}
\newextarrow{\xbigtoto}{{20}{20}{20}{20}}
{\bigRelbar\bigRelbar{\bigtwoarrowsleft\rightarrow\rightarrow}}
\newcommand{\longtoto}{\xbigtoto{}}

\newcommand{\cplx}{L}

\newcommand{\man}{M}
\newcommand{\spac}{X}

\newcommand{\Hq}{H_1}
\renewcommand{\ring}{k}

\newcommand{\HGo}{H_G^{\mr{odd}}}
\newcommand{\HGe}{H_G^{\mr{even}}}
\newcommand{\Km}{K^-}
\newcommand{\Kp}{K^+}
\newcommand{\Kpm}{K^\pm}
\newcommand{\jpm}{j^\pm}
\newcommand{\npm}{n_\pm}
\newcommand{\p}{e}
\newcommand{\ppm}{\p_\pm}
\newcommand{\pppm}{p_\pm}
\newcommand{\qpm}{q^\pm}
\newcommand{\wpm}{w_\pm}

\renewcommand{\HH}{\H_{H}}
\newcommand{\HKm}{\H_{\Km}}
\renewcommand{\HKp}{\H_{\Kp}}
\newcommand{\HKpm}{\H_{\Kpm}}
\newcommand{\HKz}{\H_{K_0}}
\newcommand{\HHz}{\H_{H_0}}

\newcommand{\MVS}{Mayer--Vietoris sequence\xspace}

\usepackage{amssymb}
\usepackage{graphics}
\usepackage{latexsym}
\usepackage{amsmath}
\usepackage{amssymb}
\usepackage{amscd}
\usepackage{multicol}
\usepackage[cmtip,matrix,arrow]{xy}
\usepackage{pst-node}
\usepackage{tikz-cd} 

\newcommand{\RR}{\mathbb R}

\begin{document}

\title{The equivariant cohomology ring of a cohomogeneity-one action}
\author{Jeffrey D. Carlson,
		Oliver Goertsches,
		Chen He, and
		Augustin-Liviu Mare}

\maketitle

\begin{abstract}
	We compute the rational Borel equivariant cohomology ring
	of a cohomogeneity-one action of a compact Lie group.
\end{abstract}

\section{Introduction}\label{sec:intro}

Cohomogeneity-one Lie group actions---%
that is, those whose orbit space is one-dimensional---%
form an intensively-studied class of examples
which are a 
next natural object of study after homogeneous actions.
In lieu of a necessarily incomplete attempt to summarize the vast geometric literature surrounding these actions,
we content ourselves with a gesture toward
the substantial bibliography to be found in the recent classificatory work of Galaz-Garc{\'i}a and Zarei~\cite{galazgarciazarei2015}.
%

Given the prominence of these actions,  
it is natural to wonder 
what can be said of their equivariant cohomology. 
Due to earlier work of two of the authors~\cite{goertschesmare2014,goertschesmare2017},
it is known the rational Borel equivariant cohomology ring
is Cohen--Macaulay, and structure theorems for this 
ring have been worked out in special cases~\cite[Cor.~4.2, Props.~5.1, 5.10]{goertschesmare2014}
along with topological consequences for 
the manifold acted on.
In this work, we describe the equivariant cohomology rings of a certain broad class of cohomogeneity-one actions (delineated precisely in the discussion after \Cref{thm:GGZalexandrov}), obtaining more explicit expressions in the case of actions on manifolds.

In the most interesting case, where the space $\spac$
acted on by a compact, connected Lie group $G$ is a manifold and 
the orbit space $\spac/G$ is a closed interval,
\Cref{thm:MostertGGZ}, 
due mostly to Mostert,
implies $\spac$ can be written up to $G$-equivariant homeomorphism as the double mapping cylinder
$\smash{G/K^- \union \big([-1,1]\x G/H\big) \union G/K^+}$ 
of a pair of quotient maps $G/H \rightrightarrows G/\Kpm$ for some closed subgroups $\defm H \leq \defm \Kpm$ of $G$.
By work of Galaz-Garc{\'i}a--Searle (\Cref{thm:GGZalexandrov}),
the same holds if $X$ is an Alexandrov space.
As such, the equivariant Mayer--Vietoris sequence is applicable
to the cover of $\spac$ by the preimages of two subintervals of $\spac/G$.\footnote{\ 
The non-equivariant Mayer--Vietoris sequence of the same cover has also long been used to study 
such spaces~\cite{grovehalperin1987,hoelscher2010homology,escherultman2011}.
}
As the equivariant cohomology of the restricted actions is well-known,
this approach would in general recover the additive structure
up to an extension problem,
but in our case, surprisingly, we are able to determine the 
entire ring structure. 
This is \Cref{thm:mainH}. 
In \Cref{sec:special}, 
we prove more explicit formulas depending on the numbers $\dim \Kpm/H \pmod 2$ in the case $\spac$ is a manifold $\man$,
so that $\Kpm/H$ are homology spheres.
In the following result and throughout,
$\defm{\H_\Gamma} \ceq H^*_\Gamma(*\mspace{2mu};\Q) = H^*(B\Gamma;\Q)$
will denote the Borel $\Gamma$-equivariant cohomology ring of a point with rational coefficients. In fact, all cohomology will take rational coefficients
unless explicitly specified otherwise.

 \begin{restatable}{theorem}{Hodd}\label{thm:Hodd}
 	Let $\man$ be the double mapping cylinder
	of the quotient maps $G/H \longtoto G/\Kpm$ for closed subgroups $ H <  \Kpm$ of a compact Lie group $G$
	such that $K^\pm/H$ are homology spheres.\\
 	\vspace{-1em}
 	\ \\
\nd	(a) 
 	Assume $K^+/H$ is odd-dimensional
 	and $K^-/H$ even-dimensional,
 	and the bundle $BH \lt BK^+$ is orientable.
 	Then we have an $\HG$-algebra isomorphism
 	\[
 	\HG M \iso \HKm \+ \p \HH[\p] < \HH[\p],  
 	\]
 	where $\HKp \iso \HH[\p]$
 	for a certain class $\p \in H_{\Kp}^{1 + \dim \Kp/H}$,
 	the product $\HKm \x \HH[\p] \lt \HH[\p]$ is determined by the injection
 	$\HKm \longmono \HH$,
 	and the $\HG$-module structure is induced
 	by the inclusions $\Kpm \longinc G$.
 	If $\Kp/H$ is a sphere, then $\p$ is 
 	the Euler class of the sphere bundle
 	$BH \lt B\Kp$.
 	\\ 
 	\vspace{-1em}
 	\ \\
 	\nd	 (b) 
 	Assume that both $K^{\pm}/H$ are odd-dimensional
 	and the bundles $BH \lt B\Kpm$ are both orientable.
 	Then we have an $\HG$-algebra isomorphism
 	\[
 	H^*_G \man \, \iso \,\quotientmed{\HH[\p_-,\p_+]\,}{(\p_- \p_+)},
 	\]
 	where $\HKpm \iso \HH[\ppm]$
 	for classes $\ppm \in H_{\Kpm}^{1 + \dim \Kpm/H}$
 	and the $\HG$-module structure is induced
 	by the inclusions $\Kpm \longinc G$.
 	If $\Kpm/H$ is a sphere, then $\ppm$ is 
 	the Euler class of the sphere bundle
 	$BH \lt B\Kpm$.
 \end{restatable}

 In the event both spheres are even-dimensional, the generators
 of the Weyl groups $W(\Kpm)$ with respect 
 to a shared maximal torus generate a dihedral 
 subgroup of the automorphisms of this torus, of order $2k$.
 It is this $k$ that figures in the following result.
 
 \begin{restatable}{theorem}{Heven}\label{thm:Heven} 
 	Let $\man$ the double mapping cylinder
 	of the quotient maps $G/H \longtoto G/\Kpm$ for closed subgroups $ H <  \Kpm$ of a compact Lie group $G$
 	such that $K^\pm/H \homeo S^{2\npm}$ 
 	are even-dimensional spheres,
 	and the bundles $BH \lt B\Kpm$ are both orientable. 
 	Then the number $k(n_-+n_+)$ is even,  
 	and we have an $\HG$-algebra isomorphism
 	\[
	 	\HG \man \iso 
	 	(\im \rho_+^* \inter \im \rho_-^*) 
			\ox 
		H^*S^{k(n_-+n_+)+1},
 	\]
 	where the injections $\rho_\pm^*\:\HKpm \lt \HH$
 	are induced by the inclusions $H \longinc \Kpm$
 	and the $\HG$-module structure is induced
 	by $\Kpm \longinc G$.
 \end{restatable}

\nd Cohomogeneity-one actions whose orbit space is $S^1$
arise as mapping tori of right translations $r_n$ of homogeneous 
spaces $G/K$ by elements $n \in N_G(K)$, 
and this case admits a parallel but much more easily-proved 
statement we discuss in \Cref{sec:mappingtori}.

\begin{restatable}{theorem}{circle}\label{thm:circle}
Let $\man$
be the mapping torus
of the right translation by $n \in N_G(K)$
on the homogeneous space $G/K$
of a compact Lie group $G$.
Then one has $\H S^1$- and $(\HG \ox \H S^1)$-algebra isomorphisms
\eqn{
	\H \man &\iso \H(G/K)^{\ang{r_n^*}} \ox \H S^1,\\
	\HG \man &\iso \H(BK)^{\ang{r_n^*}} \ox \H S^1
}
respectively, where the $\H S^1$-module structure
is given by pullback from $M/G$ in both cases
and the $\HG$-algebra structure is induced by
the inclusion $K \longinc G$.
\end{restatable}

The unexpectedly great utility of the \MVS
in our situation results from an additional structural feature 
of the sequence that seems not to be frequently noted,
namely the fact that the connecting map
preserves a module structure over the cohomology ring 
of the whole space. 
This result is proved in Section \ref{sec:MV}.

\medskip

\nd \emph{Acknowledgments.} 
The authors would like to thank the referee for careful proofreading, 
for suggesting a reference, 
and for making an important correction to their statement of \Cref{thm:MostertGGZ}.
The first author would like to thank Omar Antol{\'i}n Camarena 
for helpful conversations  
and the National Center for Theoretical Sciences (Taiwan)
for its hospitality during a phase of this work.

\section{The \MVS}\label{sec:MV}

%

Let $G$ be a topological group, $\spac$ a $G$-space,
and $A,B \sub X$ two $G$-invariant subsets whose 
interiors cover $\spac$.
The rings $\HG A \x \HG B$ and $\HG(A \inter B)$ in the \MVS inherit an $\HG \spac$-module structure
by restriction. It is clear the restriction maps between
these rings are $\HG \spac$-module homomorphisms and we claim the connecting map is as well. It is enough to prove the analogous result for singular cohomology,
as the Mayer--Vietoris sequence in equivariant cohomology is just the Mayer--Vietoris
sequence of the associated cover $(A_G,B_G)$ of the homotopy orbit space $\spac_G$.

\bprop\label{thm:MVconn}
Let $\spac$ be a topological space,
$A$ and $B$ a pair of subspaces whose interiors cover $\spac$,
and $\ring$ any commutative ring with unity.
Then the connecting map ${\d_{\mr{MV}}}\: \H(A \cap B;\ring) \lt \H(\spac;\ring)[1]$ in the \MVS of $(X;A,B)$
is a homomorphism of $\H(\spac;\ring)$-modules.
\eprop
\bpf
The covering hypothesis makes the inclusion $C_*(A) + C_*(B) \longinc C_*(X)$
of singular chain complexes
a quasi-isomorphism,
whose dual
$C^*\mn(X;\ring) \longepi \Hom_\Z\!\big(C_*(A) + C_*(B),\ring\big) \eqc \defm{\cplx}$
is then again a quasi-isomorphism.
The Mayer--Vietoris sequence in cohomology is the long exact sequence 
arising from the short exact sequence of cochain complexes
\[
	0 \to 
	\cplx \os{\defm i}\lt 
	C^*(A;\ring) \x C^*(B;\ring) \os{\defm j}\lt 
	C^*(A \inter B;\ring) \to 
	0,
\]
where $i(c) = (c|_A,c|_B)$ and $j(c_A,c_B) = c_A|_{A \inter B} - c_B|_{A \inter B}$.
To define the connecting map, 
given a homogeneous cocycle $z \in Z^q(A\inter B;\ring)$,
one selects cochains $c_A \in C^q(A;\ring)$ and $c_B \in C^q(B;\ring)$
such that $j(c_A,c_B) = z$,
and then $\d_{\mr{MV}}[z] \in H^{q+1}(\spac;\ring)$ 
is the class represented by
the unique element $z'$ of $Z^{q+1}(\cplx)$
such that $i(z') = (\d c_A, \d c_B)$.
Now given a class in $H^p(\spac;\ring)$,
the restrictions of some representative $x$ to $A,B$, and $A \cap B$
factor through the representative $x|_{\cplx} \in Z^p(\cplx)$.
Observe that $x|_{A \inter B} \cup z = j(x|_A \cup c_A,x|_B \cup c_B)$.
But since $\d x = 0$, we have 
$\big(\d(x|_A \cup c_A),\d(x|_B \cup c_B)\big) = 
({x|_A \cup \d c_A},{x|_B \cup \d c_B}) = i(x|_L \cup z')$,
so $\d_{\rm{MV}}\big([x] \cup [z]\big) = [x] \cup \d_{\rm{MV}}[z]$ as claimed.
\epf

\brmk
This feature turns out not to be specific to Borel equivariant cohomology,
but applies to multiplicative cohomology theories in general, 
and the first author proves this fact and an extension of \Cref{thm:mainH}
to other cohomology theories in an accompanying paper~\cite{carlson2018cohomogeneityK},
along with analogues of the other main results in equivariant K-theory.
\ermk

\section{Equivariant cohomology rings}\label{sec:main}

To deploy Proposition \ref{thm:MVconn} as promised,
we need the structure theorem for cohomogeneity-one actions on manifolds.

\begin{figure}
	\caption{Schematic of a double mapping cylinder}%
	\label{fig:cylinder}%
	\centering{
		\includegraphics[height=4.5cm]{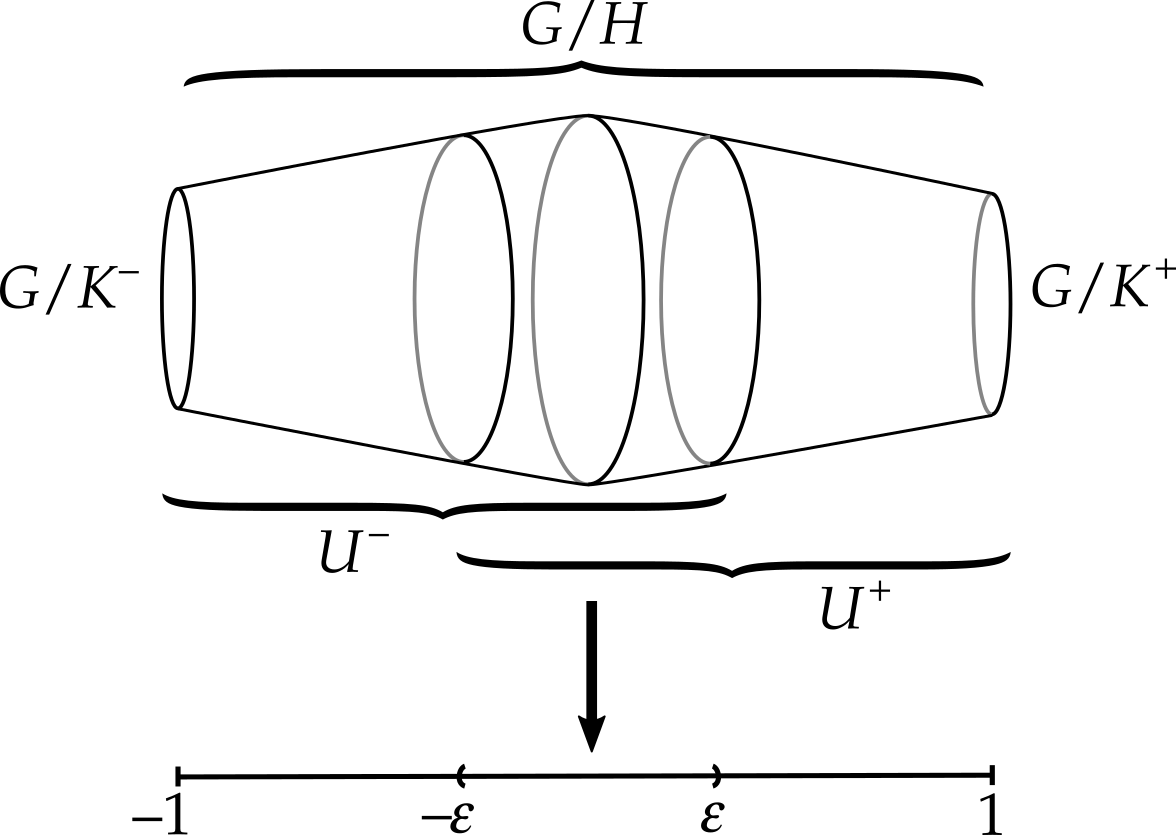}
	}
\end{figure}

\bthm[{\cite[Thm.~4]{mostert1957cohomogeneity}%
	\cite{mostert1957errata}%
	\cite[Thm.~A]{galazgarciazarei2015}}]\label{thm:MostertGGZ}
Let $G$ be a compact, connected Lie group acting continuously
with cohomogeneity one on a connected 
topological manifold $\man$ without boundary. 
Then $\man$ is, up to $G$-equivariant homeomorphism, as follows.
\bitem
\item If $\man/G \homeo (-1,1)$, there is a closed subgroup $K \leq G$
such that $\man \homeo (-1,1) \x G/K$.
\item If $\man/G \homeo S^1$, there are a closed subgroup $K \leq G$
and an element $n\in N_G(K)$ such that 
$\man$ is the mapping torus of the right translation
of $G/K$ by $n$. 
(The equivariant homeomorphism type of the resulting space depends only on the class
of $n$ in $\pi_0 N_G(K)$.)
\item If $\man/G \homeo [0,1)$, there are closed subgroups $H < K \leq G$
such that $K/H$ is either a sphere $S^n$ 
or the Poincar\'{e} homology sphere $P^3$
and $\man$ is the open mapping cylinder 
$\smash{G/K \underset\pi\union \big([0,1) \x G/H\big)}$ 
of the projection $\pi\: G/H \lt G/K$.
\item If $\man/G \homeo [-1,1]$, 
there are closed subgroups $H < K^\pm \leq G$
such that each of $K^\pm/H$ is either a sphere $S^n$ 
or the Poincar{\'e} homology sphere $P^3$
and $\man$ is the double mapping cylinder 
\[
{G/K^- 
	\underset{\pi^-}\union 
\big([-1,1]\x G/H\big) 
	\underset{\pi^+}\union 
G/K}
\]
of the projections 
$\pi^\pm\: G/H \longtoto G/\Kpm$.
\eitem
Conversely, these constructions yield only cohomogeneity-one
$G$-actions on manifolds.
In the cases where $M/G$ has boundary,
$M$ admits a $G$-invariant smooth structure if and only if 
no isotropy quotient $K/H$ or $\Kpm/H$ is $P^3$.
\ethm

Before proceeding, we note the noncompact cases are trivial
for our purposes,
since in these cases $\man$ equivariantly deformation retracts to 
the cohomogeneity-zero case $G/K$.
There is a similar classification of cohomogeneity-one actions on 
closed Alexandrov spaces.

\bthm[{\cite[Thm.~A]{galazgarciasearle2011}}]\label{thm:GGZalexandrov}
Let $G$ be a compact Lie group 
 acting effectively and isometrically with cohomogeneity one on a
closed (i.e., compact and without boundary) 
Alexandrov space $\spac$.
Then $\spac$ is, up to $G$-equivariant homeomorphism, as follows.
\bitem
\item 
If $\spac/G \homeo S^1$,
there are a closed subgroup $K \leq G$
and an element $n\in N_G(K)$ such that 
$\spac$ is the mapping torus of the right translation
of $G/K$ by $n$
(and hence, by \Cref{thm:MostertGGZ}, a smooth manifold).
\item 
If $\spac/G \homeo [-1, 1]$,
there are closed subgroups $H <  K^\pm \leq G$
such that $K^\pm/H$ are positively-curved homogeneous spaces
and $\spac$ is the double mapping cylinder 
of the projections 
$G/H \longtoto G/\Kpm$.
\eitem
Conversely, these constructions yield only 
cohomogeneity-one $G$-actions on Alexandrov spaces.
\ethm

Thus our \Cref{thm:mainH} will apply more generally than just to manifolds.
Because of these two classification results, 
it is reasonable to focus our attention in the rest of the paper 
on cohomogeneity-one actions of the following types:

\bitem
\item the mapping torus of a right translation 
on a homogeneous space $G/K$ 
or
\item
the double mapping cylinder of a span of projections
$G/K^- \from G/H \to G/K^+$.
\eitem

%
%

\subsection{Mapping tori}\label{sec:mappingtori}
By \Cref{thm:MostertGGZ},
if the orbit space of a cohomogeneity-one action is a circle,
the space in question can be assumed to be a manifold $M$,
the mapping torus of the right-translation $\defm{r_n}$ of some element $\defm n \in N_G(K)$ on $G/K$,
and hence actually a smooth manifold $M$.

\blem\label{thm:mappingtorus}
Let $Y$ be a topological space, 
$\varphi$ a self-homeomorphism 
of $Y$ such that some finite power $\varphi^\ell$ is 
homotopic to $\id_Y$,
and $\spac$ the mapping torus of $\varphi$.
Then \[
\H \spac \iso \xt{}{\H (Y)^{\ang{\varphi^*}}}{\H S^1}.
\] 
\elem
\bpf
Note that $\spac$ admits an $\ell$-sheeted cyclic covering 
by the mapping torus of $\varphi^\ell$,
which is homeomorphic to the mapping torus $Y \x S^1$
of the identity. 
The covering action is conjugate to a $\Z/\ell$-action on 
$Y \x S^1$ under which $1 + \ell\Z$ acts, up to homotopy,
as $(y,\t) \lmt \big(\varphi(y),\t+ \frac {2\pi}\ell\big)$,
which, rotating the circle component,
is in turn homotopic to $(y,\t) \lmt \big(\varphi(y),\t\big)$.
A standard lemma on the transfer map~\cite[Prop.~3G.1]{hatcher}
then gives
\[
\H \spac \iso \H(Y \x S^1)^{\Z/\ell} \iso \H(Y)^{\ang{\varphi^*}} \ox \H S^1.
\qedhere
\]
\epf

%
%

\circle*
\bpf
We continue to denote by
$r_n \mspace{-1.5mu}\: G/K \lt G/K$ 
right multiplication 
by an element $n\in N_G(K)$. 
As $N_G(K)$ is compact, it has finitely many path-components,
so some power $n^\ell$ lies in the path-component of the identity
and hence the corresponding power $r_n^\ell$ is homotopic to the 
identity. 
Applying \Cref{thm:mappingtorus} to $\man$ 
gives the first displayed isomorphism
and applying it to $\man_G$ gives the second.
%
\epf

\subsection{Double mapping cylinders}

Let $G$ be a compact Lie group and $H \leq K^\pm \leq G$
any closed subgroups.
Then the double mapping cylinder $\spac$ of $\pi^\pm\: G/H \longtoto G/\Kpm$
admits the obvious invariant open cover
by the respective inverse images $U^-$ and $U^+$
of the subintervals $[-1,\e)$ and $(-\e,1]$, for some small $\e > 0$, 
of $\spac/G \homeo [-1,1]$ depicted in \Cref{fig:cylinder}.
Their intersection
$W = U^- \cap U^+$ equivariantly deformation retracts to $G/H$ and 
$U^\pm$ to $G/\Kpm$
in such a way that the inclusions $W \longinc U^\pm$
correspond to the projections $\pi^\pm$.
Now we can apply \Cref{thm:MVconn}.

\begin{restatable}{theorem}{mainH}\label{thm:mainH}
	Let $\spac$ be the double mapping cylinder
	of the projections 
	$\pi^\pm\: G/H \longtoto G/\Kpm$.
	Then one has a graded $\HG$-algebra and a graded $\HH$-module isomorphism,
	respectively:
	\[
	\HGe \spac \,\iso\, \xu{\HH}{\HKm}{\HKp}
	,\qquad
	\HGo  \spac \,\iso\mspace{1mu} 
	\Big(\frac{\HH}{\im \rho^*_- + \im \rho^*_+}\Big)[1],
	\]
	where $\defm{\rho^*_\pm}\: \HKpm \lt \HH$ are induced by the inclusions 
	$H \longinc \Kpm$ and $\HKm \x_{\HH} \HKp$ denotes 
	the fiber product.\footnote{\ 
	That is, $\HKm \x_{\HH} \HKp < \HKm \x \HKp$
	is the subring of pairs $(x_-,x_+)$ 
	such that $\rho^*_{-}(x_-) = \rho^*_{+}(x_+)$.
}
	The multiplication of odd-degree elements is zero,
	and the product 
	$\HGe \spac \x \HGo \spac \lt \HGo \spac$
	descends from the multiplication of $\HH$
	in that
	\[
	(x_-,x_+)\.\bar q = \ol{ \rho^*_{-}(x_-) \.q }.
	\]
	for $(x_-,x_+) \in \xu{\HH}{\HKm}{\HKp}$ and 
	$\bar q \in H_G^{\mr{odd}}\spac$
	the image of $q \in \HH$.
\end{restatable}

\bpf
For any $\G \leq G$
we have homeomorphisms  
$(G/\G)_G = EG \x_G G/\G \homeo EG/\G = B\G$.
As $\smash{H^*_{ K^\pm}}$ and $\smash{H^*_H}$ are concentrated in even degree, 
the \MVS of the cover just discussed
reduces to a four-term exact sequence
\[
	0 \to
	\HGe\spac \os{\defm{i^*}}\lt
	\HKm \x \HKp \lt 
	\HH \lt
	\HGo\spac \os{i^*}\to
	0,
\]
so that $\HGe\spac$ is the kernel and
$\HGo\spac$ the cokernel of the middle map
$(x_-,x_+) \lmt \rho_+^*(x_+) - \rho_-^*(x_-)$.
The multiplicative structure on $\HGe$
follows from the fact $i^*\: \HG \spac \lt \HKm \x \HKp$
is a ring map,
the description of the product 
$\HGe \spac \x \HGo \spac \lt \HGo \spac$
follows from \Cref{thm:MVconn},
and the fact the product 
$\HGo \spac \x \HGo \spac \lt \HGe \spac$
is zero follows from the observation
$i^*$ is injective on $\HGe\spac$ and yet
$i^*(xy) = i^*(x)i^*(y) = 0$
for any elements $x,y \in \HGo\spac$.
\epf

\bex
Let $G = \O(n)$ with $K = K^\pm = \O(3)$ and $H = \O(2)$ block-diagonal.
We have $\HK \iso \Q[p_1] \iso \HH$,
where $p_1$ is the first Pontrjagin class of the tautological 
bundle over the infinite Grassmannian $\Gr(3,\Ri) = B\O(3)$,
so $\HG \spac \iso \Q[p_1]$.
\eex

\bex
In the situation where $G = K^\pm$,  
the resulting double mapping cylinder is just the unreduced 
suspension $S(G/H)$.
One has 
\eqn{
\phantom{\mbox{and}}\qquad
\HGe S(G/H) = \HG, &\qquad
\HGo S(G/H) = \quotientmed{\HH\,}{\im(\mspace{-1mu}\HG \to \HH)} [1].
}

\eex

\section{Maps of classifying spaces}\label{sec:mapBG}

In the event the cohomogeneity-one double mapping cylinder is a manifold $\man$, 
we will presently see that more precise descriptions of $H^*_G \man$ can be obtained depending on the dimensions
of the isotropy spheres $\Kpm/H$,
subject to an orientation hypothesis
on the left action of $\Kpm$ in case these groups are disconnected,
and these descriptions depend crucially on 
the structure of $\HH$ as a module over $\HKpm$. 

\blem\label{thm:simple}
Let $H < K$ be compact Lie groups such that $K/H$ is a
homology sphere. 
The $K/H$-bundle $\smash{\rho\:BH \lt BK}$
is orientable if and only if 
left multiplication on $K/H$ by any element of $K$ 
induces the identity in cohomology. 
\elem
\bpf
Recall orientability of a fiber bundle is defined as triviality 
of the action of the fundamental group of the base on the cohomology of the fiber.
To determine the action-up-to-homotopy of a class of $\pi_1(BK,e_0 K)$
on the fiber,
lift a representative loop $\eta$ to a 
path $\wt\eta$ in $EK$
starting at $e_0$ and ending at some $e_0 k_1$.
Then for any $e_0 kH$ in the fiber $\rho\-(e_0 K)$,
the path $\wt\eta kH$ lifts $\eta$ to $BH$,
starting at $e_0kH$ and ending at $e_0 k_1 k H$,
which we may define to be $\eta \.e_0kH$.
Under the identification $\rho\-(e_0 K) \homeo K/H$
given by $e_0 kH \bij kH$,
this is just the action of $\pi_0 K$ 
induced by the defining homogeneous action of $K$ on $K/H$.
The generator $1 \in H^0(K/H)$ is invariant trivially,
so the action is trivial in cohomology 
if and only if the fundamental class $[K/H]$ is fixed by the $\pi_0 K$-action.
\epf

\brmk\label{rmk:connected}
Particularly, this rules out the case $K/H \homeo S^0$ going forward.
\ermk

\begin{proposition}[Cf. Goertsches--Mare~{\cite[Prop.~3.1]{goertschesmare2014}%
\cite[Prop.~4.2]{goertschesmare2017}}]\label{thm:odd}
Let $H < K$ be compact Lie groups 
such that $K/H$ is a homology sphere of odd dimension $n$
and $BH \lt BK$ is an orientable $K/H$-bundle.
Then $\rho^*\:\HK \lt \HH$ is a surjection
and can be written 
\[
	\HH[\p]
		\xrightarrow{\p \mapsto 0} 
	\HH,
\]
where $\defm\p \in H^{n+1}_H$ is the generalized Euler class of the bundle
$K/H \to BH \to BK$.
\end{proposition}
\bpf
Let $K_0$ be the identity component of $K$
and write $\defm{\Hq} = H \cap K_0$,
so that $K_0/\Hq \lt K/H$ is a homeomorphism.
We claim $\H_{\Hq}$ is a polynomial ring.
This follows if $K/H$ is a sphere of dimension at least $2$ 
since then $K_0/\Hq \homeo K/H$ is both simply-connected and covered by $K_0/H_0$,
forcing $H_0 = \Hq$. 
If $K/H$ is homeomorphic to $S^1$ or the Poincar{\'e} homology sphere $P^3$,
then we still know
$H^*_{\Hq} \iso (\HHz)^{\Hq/H_0}$, 
but the action of $\Hq/H_0$ on $\HHz$ 
is trivial because it is known~%
\cite[Pf., Prop. 3.1]{goertschesmare2014}%
\cite[Pf., Prop. 4.2]{goertschesmare2017}
that $H_0$ is normal in $K_0$
in these cases, so the action of $\Hq/H_0$
is the restriction of an action of $K_0/H_0$,
which is homotopically trivial since $K_0/H_0$ is path-connected.\footnote{\ 
	To make this account self-contained, the proof of normality is thus.
	The transitive action of $K_0$ on $K_0/H_1$ induces 
	a map $\defm \l\:K_0 \lt \Homeo K_0/H_1$
	whose image, which acts effectively by definition, 
	can only be $S^1$ itself if $K_0/H_1 \homeo S^1$
	and $\SO(3)$ if $K_0/H_1 \homeo P^3$~\cite[Thm.~1.1]{bredon1961homogeneous}.
	As $\ker \l$ stabilizes all points, it is in particular contained in $\Hq$.
	The stabilizer of the coset $1\Hq \in K_0/\Hq$ 
	under the effective action of $\im \l$ is $\l(\Hq) \iso \Hq/\ker \l$,
	which must be finite since $\im \l$ is of rank one,
	so $\ker \l$ is of finite index in $H_1$;
	particularly, its identity component must be $H_0$.
	Since $\ker \l$ is normal in $K_0$ by definition,
	so also must be $H_0$.
	}

We now consider the map of $K/H$-bundles
\quation{
	\label{eq:BKsquare}
	\begin{aligned}
		\xymatrix@R=3em@C=1.5em{
			EK/\Hq \ar[r]\ar[d]  & EK/H\ar[d]\\
			EK/K_0 \ar[r] 		& EK/K.
		}
	\end{aligned}
}
The left map is equivariant with respect to the right $H$-action,
inducing an effective right action of 
$\defm \pi \ceq H/\Hq \iso \pi_0 K$
such that the right map is the quotient,
and so we may identify \cite[Prop.~3G.1]{hatcher} 
the map $\HK \lt \HH$ with the map of invariants
$(\HKz)^\pi \lt (\H_{\Hq})^\pi$.
Now let us consider a portion of the induced map of 
generalized Gysin sequences~\cite[{\SS}3.7]{mimuratoda}:
\[
\xymatrix{
\HK \ar[r]^(.45){\cup e}  \ar@{ >->}[d]& H^{*+n+1}_{K} \ar[r]^{\rho^*} 
\ar@{ >->}[d]& H^{*+n+1}_{H} \ar@{ >->}[d]\\
\HKz \ar[r]^(.45){\cup e} & H^{*+n+1}_{K_0} \ar[r]&H^{*+n+1}_{H_1}.
}
\]
The commutativity of the diagram implies the identification
$\HK \iso (\HKz)^\pi$ takes the one generalized Euler class to the other,
or in other words that the class in the lower sequence is $\pi$-invariant.
Since $\HKz$ is a polynomial ring, multiplication by $e$ 
is injective, so the horizontal maps before and after are zero,
so this is actually an inclusion of short exact sequences.
As the image of multiplication by $e$ is precisely the principal ideal $(e)$,
we obtain an $\pi$-equivariant isomorphism $\H_{H_1} \iso \HKz / (e)$.
As $\HKz \lt \H_{\Hq}$ is a
$\pi$-equivariant surjection between graded polynomial rings over $\Q$
on $n+1$ and $n$ generators, respectively, 
whose kernel is generated by the $\pi$-invariant element $\p$,
\Cref{thm:polysurj} applies 
to yield a $\pi$-equivariant isomorphism 
$\H_{\Hq}[\p] \isoto \HKz$.
This restricts to an isomorphism of $\pi$-invariants
$\smash{\HH[\p] \isoto \HK}$.
\epf

\brmk\label{thm:eulerodd}
When $K/H$ is a sphere,
the generalized Euler class 
featuring in \Cref{thm:odd}
is well known to be the standard Euler class of a sphere bundle%
~\cite[Thm.~5.17, pp.~145--6]{mimuratoda}.
If $K/H \homeo P^3$, 
on the other hand,
from the fact that the action $K \lt \Homeo P^3$
factors through an $\SO(3)$ subgroup~\cite[Thm.~1.1]{bredon1961homogeneous},
one can associate to $BH \lt BK$
the principal $\SO(3)$-bundle $\xi\: EK \x_K \SO(3) \lt BK$
and show $e$ is $60$ times the first Pontrjagin class $p_1(\xi)$.
\ermk

\brmk
The persistent orientability hypothesis is necessary;
we will see in \Cref{rmk:eg:mostert} that 
\Cref{thm:spherefiber} fails without this hypothesis.
For now, consider the case of $K = \O(2)$ 
and $H$ a subgroup of order $2$ generated by an element $h$
of determinant $-1$.
Then for $z \in \SO(2)$
we have $h\. zH = z^{-1}H$.
We have $\HKz = \H_{\SO(2)} = \Q[c_1]$ 
and $c_1 = \p$ in the notation of the proof 
since $\smash{\HHz = \H_{\{1\}} = \Q}$.
The proof would go through if we had $\p = c_1 \in \HK$,
but we do not;
in fact $\H_{\O(2)} \iso \Q[c_1^2]$,
where $c_1^2$ is represented by the first Pontrjagin
class $p_1$ of the tautological $2$-plane bundle over
the infinite Grassmannian 
$G(2,\R^\infty) = B\O(2)$~\cite[Thm.~3.16(a)]{VBKT}.
Thus, although it is incidentally true in this case 
that $\HK \iso \HH[c_1^2]$, the proof 
of \Cref{thm:odd} cannot possibly go through.
\ermk

\blem\label{thm:polysurj}
Suppose $A$ is a graded polynomial ring in finitely many variables over $\Q$
equipped with an action of a finite group $\pi$ fixing 
the field of constants $\Q$, 
and $x$ is a $\pi$-invariant homogeneous element of $A$
such that $B = A/(x)$ is again a polynomial ring. 
Then there is a 
$\pi$-invariant graded $\Q$-subalgebra 
$A' < A$ 
such that $A = A'[x]$ and $\smash{A' \inc A \xepi{\defm\vp} B}$ is a ring isomorphism.
\elem
\bpf
We consider $A$ and $B$ as modules over the group ring $\Q\pi$.
It is easy to see that any $\Q\pi$-module $C$ 
complementary to $(x)$
will be taken bijectively and $\Q\pi$-linearly to $A/(x)$
by $\vp$, so we just need to show 
such a complement can be chosen to be a ring
and generators of this ring can be chosen such that together with $x$
they form a set of $\Q$-algebra generators for $A$. 
We write $\defm Q B = B^{\geq 1} / B^{\geq 1}\.B^{\geq 1}$
for the graded $\Q$-module of indecomposables,
in this case a free module.
As $QB$ is a finite-dimensional $\pi$-representation
and the order of $\pi$ is invertible over $\Q$,
we may break $QB$ into irreducible representations of $\pi$,
which are cyclic $\Q\pi$-modules, each generated by an element $\bar b_j$.
We may lift each of these to a homogeneous element $b_j \in B^{\geq 1}$.
By construction, the union of the $\pi$-orbits of the $b_j$ 
forms a set of $\Q$-algebra 
generators for $B$.
Now let $a_j \in C$ be a $\vp$-preimage of $b_j$.
Then the $\Q\pi$-algebra $A'$ generated by the $a_j$ is taken 
bijectively onto $B$ by $\varphi$, so it is a polynomial subalgebra
and in fact another $\Q\pi$-linear complement to $(x)$.
It is clear from the isomorphism $A' \simto B$ that 
each element of $A$ can be represented uniquely as a polynomial in $x$ over $A'$.
\epf

The case $K/H$ is even-dimensional is simpler.

\bprop[{\cite[Thm.~26.1(a)]{borelthesis}\cite[p.~1121]{samelson1941samelson}}]\label{thm:even}
Let $H \leq K$ be compact Lie groups of equal rank. 
Then $\HK \lt \HH$ is injective.
This applies particularly if $K/H$ is an
even-dimensional sphere.
In this case, 
if the bundle $\rho\: BH \lt BK$ is orientable 
and $n = \dim K/H \geq 2$,
then $\HH$ is a free $\HK$-module of rank two
on $1$ and a lift $\defm \p \in H_H^n$ of the fundamental class 
of $K/H$ under the surjection $\H(BH) \lt \H(K/H)$.
\eprop
Samelson showed the ranks are equal if $K/H$ is an even-dimensional
sphere and the injectivity statement is due to Borel.
\bpf
The covering $\pi_0 K \to BK_0 \to BK$
coming from the action of $\pi_0 K = K/K_0$
on $BK_0 = EK/K_0$,
induces an isomorphism $\HK \iso (\HKz)^{\pi_0 K}$
by a standard transfer lemma \cite[Prop.~3G.1]{hatcher}.
As $\HKz$ is a polynomial ring by Borel's theorem,
the \SSS of $K/H \to BH \to BK$ is concentrated in even degree
and so collapses. 
By \Cref{thm:simple}, 
the coefficients are simple,
so $\HH \iso \HK \ox \H(K/H) \iso \HK\{1,\p\}$ as an $\HK$-module
for some $e$ represented by $1 \ox {[K/H]}$
in the associated graded algebra.
\epf

\brmk\label{rmk:pinv}
Note that the basis $\{1,\p\}$ is preserved under the map 
\[
\HH \iso \HK \ox \H(K/H) \longmono \HKz \ox \H(K_0/H_0) \iso  \HHz
\]
induced by the map of $K/H$-bundles $(BH_0 \to BK_0) \lt (BH \to BK)$,
so $\p \in \HHz$ may be chosen $\pi_0 H$-invariant.
\ermk

\brmk
The lift $e$ in \Cref{thm:even} can be chosen to be the
pullback under $BH \lt B\SO(n)$
of the universal Euler class $e_{\SO(n)} \in H^n_{\SO(n)}$.
To see this,
note that by the classification of transitive Lie group actions on spheres%
~\cite[Ex.~7.13]{Be},
the action $K \lt \Homeo K/H$ must factor through 
a subgroup isomorphic to $\SO(n+1)$, 
sending $H$ into an $\SO(n)$ subgroup 
and hence inducing a map of $S^n$-bundles from 
$BH \to BK$ to $B\SO(n) \to B\SO(n+1)$.
Both {\SSS}s collapse at $E_2$,
and the $E_2$ map 
$\H_{\SO(n)} \ox \H S^n \lt \HK \ox \H(K/H)$
sends $1 \ox {[S^n]} \lmt 1 \ox {[K/H]}$.
But $e$ represents $1 \ox {[K/H]}$,
and since $\H_{\SO(n)} \iso \H_{\SO(n+1)}\{1,e_{\SO(n)}\}$ 
as an $\H_{\SO(n+1)}$-module,
$e_{\SO(n)}$ represents $1 \ox {[S^n]}$.

\ermk


\section{Double mapping cylinders which are manifolds}\label{sec:special}

In this section we are in the situation of 
\Cref{thm:mainH} and additionally the isotropy quotients
$\Kpm/H$ are homology spheres. 

 \subsection{The case when one of $\Kpm/H$  
 	is odd-dimensional}\label{sec:odd}
 
If $K^+/H$ is odd-dimensional, 
then $\HKp \lt \HH$ is surjective, 
so $\HGo \man$ vanishes by \Cref{thm:mainH} 
and $\HG \man = \HGe \man$
is easily described. 

\Hodd*
 \begin{proof}
%
%
 	(a)\, 
 	Since the map $\HKp \simto \HH[\p] \to \HH$ is reduction modulo $(\p)$ 
 	by \Cref{thm:odd}
 	and $\HKm \lt \HH$ is an injection by \Cref{thm:even},
 	the fiber product is the subring of $\HKm \x \HH[\p]$
 	consisting of the direct summands 
 	$\big\{(x,x) \in \HKm \x \HKm\big\}$
 	and
 	$\{0\} \x \p \. \HH[\p]$.
 	We may identify the former with
 	$\HKm < \HH < \HH[\p]$
 	and the latter with 
 	$\p \HH[\p] \idealneq \HH[\p]$,
 	and the two interact multiplicatively via the rule
 	$x\. \p f \bij (x,x)\.(0,\p f) = (0,\p fx) \bij \p fx$.
 	%
 	
 	\medskip
 	
 	(b) 
 	Using \Cref{thm:odd} to make identifications $\HKpm \iso \HH[\ppm]$
 	such that $\HKpm \lt \HH$ is reduction modulo $(\ppm)$,
 	we see the fiber product is the subring of $\HH[\p_-] \x \HH[\p_+]$ 
 	comprising the three direct summands
 	\[
	 	\big\{(x,x) \in \HH \x \HH \big\},\qquad\qquad
	 	\p_- \HH[\p_-] \x \{0\},\qquad\qquad
	 	\{0\} \x \p_+ \HH[\p_+],
 	\]
 	on which multiplication is determined by the three rules 
 	\[
	 	(x,x) \. (\p_-f_-,0) = (\p_-f_-x,0),\qquad
	 	(x,x) \. (0,\p_+ f_+) = (0,\p_+ f_+x),\qquad
	 	(\p_- f_-,0)\.(0,\p_+ f_+) = (0,0),
 	\]
 	so the map to $\HH[\p_-,\p_+]/(\p_-\p_+)$
 	sending $(x + \p_-f_-,x + \p_+ f_+)\lmt x + \p_- f_- + \p_+ f_+$
 	is a ring isomorphism.
 \end{proof}

\brmk\label{rmk:CMKodd}
The second and fourth author have shown \cite{goertschesmare2014,goertschesmare2017} 
that for any cohomogeneity-one action of a compact, connected Lie group $G$ 
on a compact, connected topological manifold $M$, 
the equivariant cohomology $H^*_G M$ is a Cohen--Macaulay 
module over $H^*_G$. 
In the special case when all the hypotheses of \Cref{thm:Hodd}
are fulfilled, 
this result can be recovered easily from that theorem. 
Concretely, in case (a), 
the equivariant cohomology 
is a direct sum of two Cohen--Macaulay modules over $H^*_G$ 
and in case (b) 
it is an algebra over $H^*_G$
finitely generated as an $\HG$-module and Cohen--Macaulay as a ring. 
\ermk

\brmk
Special cases of the actions investigated in this section are also considered
by Choi and Kuroki~\cite{kuroki2011classification,choikuroki2011}.
\ermk

\subsection{The case when both of $\Kpm/H$ are even-dimensional}

In this subsection we assume that $K^-$, $K^+$, and $H$ 
have all three the same rank,
or equivalently that $M$ is a manifold with $\Kpm/H$
even-dimensional.
We start with the special case where $K^- = K^+$.

\begin{restatable}{proposition}{spherefiber}\label{thm:spherefiber}
	Assume $\defm{K} \ceq K^+ = K^-$,
	that $K/H=S^{2n}$ is an even-dimensional sphere,
	and that the bundle $BH \lt BK$ is orientable.
	Then we have an $\H(BG;\Z)$-algebra isomorphism
	\[
	\HG(\man;\Z) \iso \H(BK;\Z) \ox \H(S^{2n+1};\Z),
	\]
where the $\H(BG;\Z)$-algebra structure is induced
by the inclusion $K \longinc G$.
\end{restatable}
\bpf
In this case 
$\man_G$ is the homotopy pushout of $BK \from BH \to BK$,
which we may write as the quotient of $[-1,1] \x BH$ 
by the relation collapsing the ends $\{\pm 1\} \x BH$ to 
one copy of $BK$ each.
There is an obvious map 
\eqn{
	\defm{\xi}\: \man_G &\longepi BK,\\
	[t, eH]	&\lmt eK, \quad t \in (-1,1),\\
	[\pm 1, eK] &\lmt eK.
}
The fiber of this map over $eK \in BK$
is the unreduced suspension $S(K/H) \homeo S^{2n+1}$,
so $\xi$ is a sphere bundle. 
We claim the \SSS of $\xi$ has simple coefficients and collapses at $E_2$.
Indeed, given $eK \in BK$ and a loop $\h$ representing 
a class in $\pi_1(BK,eK)$, if $\wt\h$ is a lift to $EK$ starting at $e$
and ending at $ek$,
a lift of $\h$ starting at $(\pm 1, eK)$  (respectively, $(t,eH)$)
ends at $(\pm 1, ekK) = (\pm 1, eK)$ (resp., $(t,ekH)$).
This action fixes the poles of the fiber $S(K/H)$
and acts as left multiplication by $k$ on each latitude $K/H$,
fixing the orientation of the latitude by \Cref{thm:simple},
so the action preserves the orientation of the fiber $S^{2n+1}$
of $\xi$ as a whole, and thus, again by \Cref{thm:simple}, 
the bundle $\xi$ is orientable.
As there are only two nonzero rows,
to see the spectral sequence collapses
at 
$E_2 \iso  \H\big(BK; \H(S^{2n+1};\Z)\big) = \H(BK;\Z) \ox \H(S^{2n+1};\Z)$,
it is enough to show $d_{2n+2}[S^{2n+1}] = 0 \in H^{2n+2}(BK)$.
But this differential must be zero because 
$\xi$ admits the section $eK \lmt [-1,eK]$,
showing $\xi^*\: \H(BK;\Z) \lt \HG(\man;\Z)$ is injective.
The collapse shows $\HG(\man;\Z) \lt \H(S^{2n+1};\Z)$
is surjective
and $\HG(\man;\Z) \iso \Ei = \H(BK;\Z) \ox \H(S^{2n+1};\Z)$
as a module over $\H(BK;\Z)$.
Thus $1 \in H_G^0(\man;\Z)$
and any preimage 
$z \in H^{2n+1}_G (\man;\Z)$ of the fundamental class 
$[S^{2n+1}] \in H^{2n+1}(S^{2n+1};\Z)$
generate
$\HG(\man;\Z)$ as an $\H(BK;\Z)$-module.

By definition, this $z$ commutes with all elements 
of $\im \xi^* \iso \H(BK;\Z)$,
and we claim it also squares to zero. 
Indeed, from \Cref{thm:mainH}, 
it is in the image of the Mayer--Vietoris connecting map from $\H(BH;\Z)$,
which factors as 
$\smash{\H(BH;\Z) \os\susp\to \wt H^{*+1}(S BH;\Z) \to \wt H^{*+1}(M_G;\Z)}$,
where $\susp$ is the suspension isomorphism
and $M_G \to S BH$ is the map that collapses each end $BK$
of the double mapping cylinder to a point~\cite[{\SS}19.1--3, pp.~146--7]{may1999concise}.
Since the cup product on $\wt H^*(S BH)$ is identically zero
and $z^2$ is the image of a square in $\wt H^*(S BH)$,
we see that indeed $z^2 = 0$.
It follows 
$\HG(\man;\Z) \iso \ext_{\H(BK;\Z)}[z] \iso \H(BK;\Z) \ox \H(S^{2n+1};\Z)$ 
as an $\H(BK;\Z)$-algebra.
\epf

\brmk\label{rmk:eg:mostert}
The orientability hypothesis in the proposition above 
is essential, as can be seen from the action of
${\rm S}{\rm O}(3)$ on $\RR P^3 \,\#\, \RR P^3$ described by Mostert%
~\cite[Thm.~7]{mostert1957cohomogeneity}. 
The isotropies are given by
$K = K^\pm = \mr{S}\big(\Orth(2) \x \Orth(1)\mspace{-1mu}\big) \iso \Orth(2)$ and 
$H=\SO(2)\times \{1\}$,
so $K/H \homeo S^0$ and the bundle is not orientable by \Cref{rmk:connected}.
If the conclusion of Proposition \ref{thm:spherefiber} held in this instance, 
we would have (with $\Q$ coefficients as usual) 
\begin{equation}\label{sop3}
	\H_{\SO(3)}(\RP^3 \,\#\, \RP^3) 
		\iso 
	\H_{\O(2)} \ox \H S^1.
	\end{equation}
But the $\SO(3)$-action at hand is equivariantly formal%
~\cite[Cor.~1.3]{goertschesmare2014}, 
so $\smash{H^*_{\SO(3)}(\RP^3 \,\#\, \RP^3)}$ is isomorphic as an
$H^*_{\SO(3)}$-module to $H^*(\RP^3 \,\#\, \RP^3)\otimes H^*_{\SO(3)}$,
which unlike $\H_{\O(2)} \ox \H S^1$ is zero in dimension $1$.
\ermk

\bex Let us now use \Cref{thm:spherefiber} to compute the equivariant cohomology 
of the action arising from the inclusion diagram 
$(G,\Km,\Kp,H) = 
	\big(
		\Sp(2), 
		\Sp(1)^2, 
		\Sp(1)^2,
		\Sp(1) \x \U(1)
	\big)$. 
For a nice treatment of this manifold we refer to P{\"u}ttmann~\cite[Sect.~4.3]{puettmann2009}. 
From the \MVS, one sees the manifold $M$ has the same integral cohomology as the direct product $S^3\times S^4$. 
By Proposition \ref{thm:spherefiber}, 
\[
	H^*_{\Sp(2)}(M; \Z) 
		\iso 
	\H\big(B\Sp(1)^2;\Z\big)
		\ox 
	H^*(S^3;\Z).
\]
Equivariant formality
of the $\Sp(2)$-action on $M$ was already known over $\Q$~\cite[Cor.~1.3]{goertschesmare2014},
but the inclusion $\Sp(1)^2 \longinc \Sp(2)$
induces an $\H\big(B\Sp(2);\Z\big)$-module isomorphism
$\H\big(B\Sp(1)^2;\Z\big) \iso \H\big(B\Sp(2);\Z\big) \ox \H(S^4;\Z)$,
so the action is actually equivariantly formal over $\Z$. 
\eex

We will now generalize the proposition above to the case when 
$K^-$ and $K^+$ are not necessarily equal.
Assume that $K^{\pm}/H = S^{2\npm}$ for $\defm\npm \ge 1$.
Let $\defm S \leq H$ be a maximal torus and $\defm\Xi$ the subgroup of
$\Aut S$ generated by the Weyl groups $W(\Kpm)$ and $W(H)$. 
The Weyl groups, and hence $\Xi$,
are all contained in the image of the conjugation map
$N_G(S) \to \Aut S$ sending $g \lmt (s \mapsto gsg\-)$.
Since image of this map is compact and $\Aut S \iso \Z^{\dim S}$ is discrete,
$\Xi$ is finite.
Because $K^\pm/H$ are even-dimensional spheres,
by \Cref{thm:even},
$\HH = (\HS)^{W_H}$ is of rank two over $\HKpm = (\HS)^{W_{\Kpm}}$,
so $W(H)$ is an index-two subgroup of each of $W(\Kpm)$ and hence normal.
It follows $W(H)$ is also normal in $\Xi$,
and we will be particularly interested in the quotient group $\Xi/W(H)$.
Note that the involutions $\defm{\wpm} \in W(\Kpm)/W(H) \iso \Z/2$ 
also generate $\Xi/W(H)$, so the latter must be a dihedral group.
Because we will be considering functions on the Lie algebra $\fs$,
in the rest of this section, cohomology will take real or complex coefficients.

\bdefn
	For a compact, disconnected Lie group $\G$ with maximal torus $T$, 
	we continue to define its
	\defd{Weyl group $W(\G)$} as $N_\G(T)/Z_\G(T)$.\footnote{\ 
		We note that this is not the only definition used
		in the literature, 
		and does not agree with the common definition
		invoking a Cartan subgroup as per 
		Segal~\cite[Def.~1.1]{segal1968representation}%
		\cite[{\SS}IV.4]{brockertomdieck}.
		}
\edefn
%
It is easy to see every component of $\G$ contains some element 
of $N_\G(T)$, 
and such elements in the same component 
differ by an element of $N_{\G_0}(T)$,
so there is a well defined action of 
$\pi_0 \G$ on the fixed point set
$\R[\ft]^{W(\G_0)}$
and one has
\[\smash{
	\H(B\G;\R) \iso \H(B\G_0;\R)^{\pi_0 \G} 
\iso \big(\R[\ft]^{W(\G_0)}\big)\mn{}^{\pi_0 \G} =  \R[\ft]^{W(\G)}
}\]
as in the connected case.

\blem\label{thm:eff} 
The action of the dihedral group $\Xi/W(H)$ on
$H^*(BH;\R)=\R[\fs]^{W(H)}$ is effective.
%
\elem

\bpf 
We show any element $\defm{c_g}\mn\: s \mapsto gsg\-$ 
of $\Xi$ that fixes $\R[\fs]^{W(H)}$ pointwise
is already in $W(H)$. 
If $\defm\Upsilon$ is the subgroup of $\Xi$ generated by 
$W(H)$ and $c_g$, then clearly
\begin{equation}\label{eq:smft}
	\R[\fs]^\Upsilon = \R[\fs]^{W(H)}.
\end{equation}
By Molien's theorem~\cite[{\SS}17-3]{kane},
given any action of a finite group $\G$ on a real vector space $V$,
the Poincar{\'e} series in $t$ of $\R[V]^\G$
is a polynomial 
in the variable $(1-t)\-$ 
with leading coefficient $1/|\G|$. 
In our case, this shows $|\Upsilon| = |W(H)|$, so $c_g$ is in $W(H)$ as claimed.
%
%
\epf

\brmk \label{rmk:riemann} 
Assume that $G$ is connected 
and $M$ smooth and equipped with 
a $G$-invariant Riemannian metric and a 
complete geodesic $\defm\g$ in $M$
meeting each orbit orthogonally. 
The \defd{Weyl group} of $\g$
is defined~\cite[{\SS}4]{palaisterng}\cite[{\SS}5]{alekseevsky1993}
to be $\defm{W(\g)} \ceq N_G(\gamma)/Z_G(\gamma)$,
where $\defm{N_G(\gamma)} < G$ is the setwise stabilizer of $\g$
and $\defm{Z_G(\gamma)}$ the pointwise. 
The account of Alekseevsky--Alekseevsky~\cite[{\SS}4,5]{alekseevsky1993} 
shows $G$ acts transitively on the set of such $\g$, 
so we may assume $\g$ passes through $(0,1H) \in (-1,1) \x G/H \subn M$; 
it can also be shown $H$ is the common stabilizer of all points of $\g$
in $(-1,1) \x G/H$,
so $W(\g) = N_G(\g)/H$,
and it follows $\g$ passes through $(\pm 1,1\Kpm) \in \{\pm 1\} \x G/\Kpm$ 
as well.
Further, there are unique involutions $\defm{\s_\pm} \in N_{\Kpm}(H)/H$
acting antipodally on the spheres $\Kpm/H$
and generating $W(\g)$ as a dihedral subgroup of $N_G(H)/H$.


One might hope from this account that the groups $W(\g)$ and $\Xi/W(H)$
are isomorphic,
but they are typically not, 
as we will see in \Cref{EG:PuettmannWeyl}.
There is at least a homomorphism from 
the one to the other, which is an isomorphism if $\rk G = \rk S$. 
To construct it,
observe first that since all maximal tori in $H$ are conjugate,
for any $g\in N_G(H)$ 
there exists an $\defm{h_g} \in H$ such that $gh_g\in N_G(S)$, 
and this specification uniquely determines the left coset 
$h_gN_H(S)$,
defining a homomorphism $g \lmt gh_g N_H(S)$
from $N_G(H) \lt \big(N_G(S) \inter N_G(H)\big) / N_H(S)$;
to see multiplicativity, note that 
for given $gH,g'H \in N_G(H)/H$,
we may make the choice $h_{gg'} = (g')\-h_g g' h_{g'}$.
It is not hard to see the kernel is $H$, 
so there is an induced monomorphism
\eqn{
	\defm\psi\: \frac{N_G(H)}{H}&
	\longmono \frac {N_G(S) \cap N_G(H)}{N_H(S)},\\
					gH &\lmt gh_g N_H(S).
} 
%
%
Following with the map 
$\defm c\: \big(N_G(S) \cap N_G(H)\big)/N_H(S) \lt N_{\Aut S}\big(W(H)\big)/W(H)$
sending $g$ to conjugation of $S$ by $g$,
we obtain a map $N_G(H)/H \lt N_{\Aut S}\big(W(H)\big)/W(H)$.
When restricted to $N_{\Kpm}(H)/H$, 
this $c \o \psi$ takes values in $W(\Kpm)/W(H)$,
so there 
is a restricted map
$\defm{\ol\psi}\: W(\g) \lt \Xi/W(H)$ as claimed.
When $\rk G = \rk S$, 
the map $c$ is injective, so 
$\ol\psi(\s_\pm) = w_\pm$
and $\ol\psi$ is an isomorphism.
%
\ermk

We henceforth write $\defm{D_{2k}}$
for the dihedral group $\Xi/W(H)$,
where $\defm k$ is the order of $w_+w_-$.
Our remaining goal is the following.

\Heven*

This will follow from an 
analysis of the action 
of ${D_{2k}}$ on $\HH$.


\blem\label{thm:wpm} 
Let $E$ be a complex representation of 
$D_{2k} = 
\ang{w_-,w_+ \mid  w_-^2, w_+ ^2, (w_+ w_-)^k}$.
Set $\defm r \ceq w_+ w_-$ 
and $\defm s \ceq w_+$ 
so that $sr = w_-$.
Write $\defm\z$ for the root of unity $e^{2\pi i / k}$
and $\defm{E_\ell}$ for the $\z^{\ell}$-eigenspace of $r$.
\bitem 
\item 
The transformations $w_-$ and $w_+$ agree on $E_0$ 
and are opposite on $E_{k/2}$ if $k$ is even.
\item
Both $w_-$ and $w_+$ exchange each
$E_\ell$ with $E_{-\ell}$.
\eitem
\elem

\bpf 
Multiplying the relation $\id = r r^{-1}$ on the left by $w_+$ yields
the key equation
\quation{\label{eq:wpmeigen}
	\phantom{E_\ell}\quad
	w_+ = w_- r\- = \z^{-\ell} w_- \quad \mbox{on} \ E_\ell.
	}
\bitem
\item
The $\ell = 0$ and $\ell = k/2$ cases of (\ref{eq:wpmeigen})
show $w_- = w_+$  on $E_0$
and $w_- = -w_+$ on $E_{k/2}$.
\item
Since $rw_- = w_+ w_-^2 = w_+$,
(\ref{eq:wpmeigen}) gives
$r w_- v= w_+ v= \z^{-\ell} w_- v$
for $v \in E_\ell$.
On the other hand, multiplying (\ref{eq:wpmeigen}) on the left by $r$ 
yields 
$r w_+ v  = \z^{-\ell} r w_- v =  \z^{-\ell} w_+ v.$
\qedhere
\eitem
\epf

From now on, we specialize \Cref{thm:wpm} 
to the case $E = \defm\HH = \H(BH;\C)$.
We write also $\defm{E_\ell^m} \ceq E_\ell \cap H_H^m$.

\blem\label{thm:eigenw}  
There exist $\smash{\defm \pppm \in H^{2\npm}_H}$ 
such that $\wpm\pppm = -\pppm$
and $\HH = \HKpm \+ \pppm \HKpm$.
\elem

\bpf 
Note that since $\HKpm = (\HH)^{\ang{\wpm}}$, 
the $1$-eigenspace of $\wpm$ is $\HKpm$.
Recall that \Cref{thm:even}
gives a $\HKpm$-basis $\{1,\ppm\}$ for $\HH$;
now $\pppm \ceq (\ppm - \wpm \ppm)/2$
is a $-1$-eigenvector for $\wpm$
and $\{1,\pppm\}$ is another $\HKpm$-module basis for $\HH$.
\epf

Let us now consider the $r$-eigenspace decompositions
$\pppm = \sum \defm{\qpm_\ell}$ 
where $q^\pm_\ell \in E_\ell$ for $\ell \in \Z/k$.
Since the $w_\pm$ interchange $E_\ell$ with $E_{-\ell}$ by \Cref{thm:wpm}, 
one finds $\qpm_{-\ell} = -\wpm\qpm_\ell$,
and specifically that $\wpm q_0^\pm = -q_0^\pm$,
and if $k$ is even, $\wpm\qpm_{k/2} = -\qpm_{k/2}$.
All told, the decomposition is
\begin{equation}\label{eq:ppmeven}
	\pppm =  
	\qpm_0 + 
	\sum_{0 \mspace{1mu}<\mspace{1mu} \ell \mspace{1mu}<\mspace{1mu} k/2} 
	(\qpm_\ell - \wpm\qpm_\ell) +
	\qpm_{k/2},
\end{equation}
where the last term is taken to be zero if $k$ is odd.

\blem\label{thm:cases} 
In (\ref{eq:ppmeven}) 
only one term is non-zero. 
Explicitly, the elements $\pppm$ each lie either in $E_0$, in $E_{k/2}$,
or in $E_\ell \+ E_{-\ell}$ for $0 < \ell < k/2$.
\elem

\bpf Each of terms $\qpm_0$, $\qpm_{k/2}$, and $\qpm_\ell - \wpm\qpm_\ell$ in (\ref{eq:ppmeven}) 
lies in the $-1$-eigenspace of
$w_\pm$ on $H_H^{2\npm}$,
but by \Cref{thm:eigenw} this is the one-dimensional $\C \pppm$, 
which meets at most one of  
$E_0$, $E_{k/2}$, and $E_\ell \+ E_{-\ell}$
nontrivially.
\epf


\blem \label{thm:trichotomy} Exactly one of the following cases obtains:
\begin{enumerate}
	\item[(i)] $k=1$ and $n_+=n_-\eqc \defm n$, 
		and one can rescale 
	$\pppm$ in such a way that $p_+ = p_-\in E_0^{2n}$.

	\item[(ii)] $k=2$ and $\pppm\in E_1$. 
	
	\item[(iii)] $k \geq 2$ and $n_+=n_- \eqc \defm n$, 
		and there is one $j$ relatively prime to $k$ 
		such that $0 < j < k/2$
		and up to rescaling,
		\[
			\pppm = q - \wpm q
		\]
		for the same single element $\defm q\in E_j^{2n}$. 
		Moreover, $\dim E_j^{2n}  = \dim E_{-j}^{2n} = 1$.
\end{enumerate}
\elem 

\bpf 
This is a case analysis following \Cref{thm:cases}.

\medskip

\nd\emph{(i) The case one of $\pppm$ lies in $E_0$}

\smallskip

Suppose $p_+ \in E_0$ and, 
for a contradiction, suppose $m$ is minimal such that
there exists a nonzero $x \in E_\ell^m$
for some $\ell$ indivisible by $k$.
Then $x - w_+ x$ is a $-1$-eigenvector of $w_+$,
hence divisible by $p_+$ by \Cref{thm:eigenw}.
We must have $x = w_+ x$
for otherwise the component of 
$(x- w_+ x)/p_+$ in 
$\smash{E_\ell^{m-2n_+}}$ would contradict minimality of $m$.
Thus $x \in E_{k/2}$ by \Cref{thm:wpm}, 
so $w_- x = - w_+ x = -x$ by \Cref{thm:wpm}
and $p_-$ divides $x$ by \Cref{thm:eigenw}.
It follows, again lest we contradict minimality of $m$,
that $x/p_- \in E_0$ and so $p_- \in E_{k/2}$. 
But then, by \Cref{thm:wpm} again,
$
w_- p_+ =
w_+ p_+ = 
- p_+
$
so $p_-$ divides $p_+$ by \Cref{thm:eigenw}
and we have 
\[
	w_+ \Big(\frac{p_+}{p_-}\Big) = 
	\frac{w_+ p_+ }{w_+ p_-} =
	\frac{-p_+}{-w_- p_-} =
	-\frac{p_+}{p_-},
\]
contradicting the fact $p_+$ has minimal degree in the $-1$-eigenspace
of $w_+$.
Thus in fact $\HH = E_0$. 
Since $r$ then acts trivially 
but \Cref{thm:eff} states the action of $D_{2k}$ is effective, 
it follows $k = 1$.
By \Cref{thm:wpm} the actions of $w_-$ and $w_+$ agree on $E_0$,
so the $\wpm$-eigenspace decompositions of $\HH$
in \Cref{thm:eigenw} coincide
and by rescaling we may assume $p_- = p_+$. 

\medskip

\nd\emph{(ii) The case one of $\pppm$ lies in $E_{k/2}$}

\smallskip

If $p_+ \in E_{k/2}$,
then in fact $\HH = E_0 \+ E_{k/2}$,
for we could otherwise produce from any nonzero homogeneous $x \in E_\ell$,
where $k/2$ does not divide $\ell$, 
an element $(x - w_+ x)/p_+$ of smaller degree
with a nonzero component in $E_{\ell - k/2}$.
It follows $r^2$ acts as the identity on $\HH$, 
and since the action of $D_{2k}$ is effective by \Cref{thm:eff},
we have $k = 2$ and $p_+\in E_1$. 
As for $p_-$, consulting \Cref{thm:cases},
it must lie in $E_0$ or $E_1$,
but if it lay in $E_0$ we would be in the previous case.

\medskip

\nd\emph{(iii) The case $\pppm$ both lie in $E_{\jpm} \+ E_{-\jpm}$ 
	for some $j^\pm$ with $0 < j^\pm < k/2$}

\smallskip

We first note that $n_+=n_-$,
for if we had, say, $n_+>n_-$, 
then by \Cref{thm:eigenw},
$w_+$ would act as the identity
on $\smash{H^{2n_-}_H}$;
but this would imply 
$\smash{p_- = q_{j^-} - w_- q_{j^-}}$ is equal to
$\smash{w_+ p_- = w_+ q_{j^-} - \z^{j^-} q_{j^-}}$,
which could only happen if $\z^{j^-} = -1$ and $j^- \equiv k/2 \pmod k$.

Next we claim $j^+ = j^-$. 
As already mentioned, the $-1$-eigenspace of $w_+$ in 
$H^{2n}_H$ is $\C p_+$, which lies in $E_{j^+} \+ E_{-j^+}$
Since the $\pm 1$-eigenspaces of $w_+$
on $E_\ell^{2n} \+ E_{-\ell}^{2n}$ are 
$\{x \pm w_+x : x \in E_\ell^{2n}\}$
and hence both are of dimension $\dim E_\ell^{2n}$,
but the $-1$-eigenspace is trivial except for $\ell = j^-$,
it follows all the other terms in 
$\Direct_{0 < \ell < k/2} (E_\ell^{2n} \+ E_{-\ell}^{2n})$ 
are zero, so $j^+ = j^- \eqc \defm j$ 
and $\dim E_j^{2n} = \dim E_{-j}^{2n} = 1$. 
Thus we may rescale to take $q_- = q_+$ as claimed.

To see that $j$ and $k$ are coprime,
first note that $\HH = \sum_i E_{i\.\gcd(j,k)}$,
for given a putative nonzero homogeneous element $x \in E_\ell$
such that $\gcd(j,k)$ does not divide $\ell$,
the component of $(x - w_+ x)/p_+$ 
in $E_{\ell-j}$ would be nonzero of smaller degree.
Thus $r^{k/\!\gcd(j,k)} \in D_{2k}$
acts trivially on $\HH$,
and since the $D_{2k}$-action is effective by \Cref{thm:eff},
it follows $\gcd(j,k) = 1$.
\epf

\blem\label{thm:prime} 
The elements $\pppm$ are prime elements of $H^*_H$.
\elem

\bpf 
Recall from \Cref{rmk:pinv} that 
the basis element $\p$ from \Cref{thm:even}
may be taken $\pi_0 H$-invariant.
Thus the same holds of the eigenvectors $\pppm = (\ppm -\wpm\ppm)/2$ 
defined in \Cref{thm:eigenw}.
It is clear in $\HHz$ that $p_\pm$ are irreducible,
for all elements of lesser degree lie in the $1$-eigenspace 
$\smash{H\mspace{-1mu}\substack{*\phantom{k}\\K_0^\pm}}$.
Since $\HHz$ is a polynomial ring, the principal ideals 
$(\pppm)$ are prime.
%
In fact, the ideals $(\pppm)$ are prime in $\HH$ as well, for
given $x,y \in H^*_{H}$ such that $xy \in (\pppm)$,
we know one of the two, say $x$, is divisible by $\pppm$ in $\HHz$.
But then as $x/\pppm$ is $\pi_0 H$-invariant as well,
$x$ is also divisible by $\pppm$ in $\HH$.
\epf

\blem\label{thm:negativeone}
Write $\defm V$ for the joint $-1$-eigenspace of $\wpm$.
Then  
\[
	V \lt \frac {E_0}{(\HKm + \HKp) \cap E_0}
	\lt \frac \HH {\HKm + \HKp}
\] 
are isomorphisms.
\elem

All unelaborated claims in the proof are
clauses of \Cref{thm:wpm}.

\bpf
Write $\defm{E_{\neq 0}} \ceq \Direct_{\ell = 1}^{k-1} E_\ell$
so that we have a direct sum decomposition $\HH = E_0 \+ E_{\neq 0}$.
This decomposition is invariant under $\wpm$,
so $\HKpm$ inherit such decompositions and 
\[
\HKm + \HKp = 
\big([\HKm \cap E_0] + [\HKp \cap E_0]\big) \+
\big([\HKm \cap E_{\neq 0}] + [\HKp \cap E_{\neq 0}]\big).
\]
Because $w_+|_{E_0} = w_-|_{E_0}$,
the first direct summand is the common $1$-eigenspace 
of $w_\pm$ on $E_0$, whose complement is $V$.
We will be done if we can show the second summand is all of $E_{\neq 0}$.
\bitem
\item If $k$ is even, 
$r=w_+ w_-$ acts on $E_{k/2}$ as multiplication by $\z^{k/2} = -1$,
so $w_+|_{E_{k/2}} = -w_-|_{E_{k/2}}$.
Thus $E_{k/2}$ decomposes as the sum of the
$1$-eigenspace $\HKp \cap E_{k/2}$ of $w_+$ on $E_{k/2}$
and its $-1$-eigenspace, which is $\HKm \cap E_{k/2}$.
\item If $0 < \ell < k/2$, 
then since $r$ acts as $\z^\ell \neq \pm 1$ on $E_\ell$,
we have $w_+ v \neq w_- v$ for nonzero $v \in E_\ell$,
so for nonzero $u,v \in E_\ell$, 
we cannot have $u + w_+ u = v + w_- v$ in $E_\ell \+ E_{-\ell}$.
Thus the $1$-eigenspaces $(\id + w_\pm) E_\ell$ of $\wpm$ are disjoint,
and since $\dim E_\ell^m = \dim E_{-\ell}^m$ for all $m \geq 0$,
their sum is all of $E_\ell \+ E_{-\ell}$.\qedhere
\eitem

\epf

We are now finally in a position to prove Proposition \ref{thm:Heven}.

\bpf[Proof of Proposition \ref{thm:Heven}]
We proceed through the trichotomy of \Cref{thm:trichotomy},
in each case applying \Cref{thm:mainH}.

\medskip

\nd\emph{(i) The case $p = p_+ = p_- \in E_0^{2n}$ and $\HH = E_0$}

\smallskip

In this case the actions of $\wpm$ coincide by \Cref{thm:wpm},
so the images $\defm \HK$ of the injections $\HKpm \longinc \HH$ agree.
Thus
\eqn{
	\HGe \man &= \HKm \cap \HKp = \HK;\\
	\HGo \man &= \frac{\HK \+ p \HK}\HK[1] \iso p\HK[1] \iso \HK[2n+1].
}

\medskip

\nd\emph{(ii) The case $k=2$ and $\HH = E_0 \+ E_1$ and $\pppm \in E_1$.}

\smallskip

We show $V = p_+ p_- \. (\HKm \inter \HKp)$
and apply \Cref{thm:negativeone};
since $\deg p_+p_- = 2n_- + 2n_+$, the result will then follow
from \Cref{thm:mainH}.
Suppose $x$ lies in $V$.
From \Cref{thm:eigenw} we know $x$
is divisible by both $p_-$ and $p_+$.  
Note that $w_\mp \pppm = \zeta^{-k/2} \wpm\pppm = \pppm$ by \Cref{thm:wpm}.
As $w_-$ 
sends $x/p_+$ to $w_-x / w_-p_+ = -(x/p_+)$,
we see $p_-$ divides the latter as well.
The quotient $w_\pm (x/p_+ p_-) = -x/{-p_+p}_- = x/p_+p_-$ 
is in the joint $1$-eigenspace $\HKm \inter \HKp$. 


\medskip

\nd\emph{(iii) The case $\pppm = q - \wpm q$ for some $q\in E_j^{2n}$.}

\smallskip

We will find an element $\defm P$ of degree $2nk$
such that $V = P \. (\HKm \cap \HKp)$
and apply \Cref{thm:negativeone}.
Since $\wpm$ generate all $w \in D_{2k}$,
we have $w\. x = \pm x$ for any $x \in V$.
Particularly, $x$ is also divisible by the $w\pppm$, 
which we explicitly enumerate.
Writing $\defm\h = \z^j$,
from \Cref{thm:wpm} we obtain  relations 
$p_- = q - \h w_+ q$
and $r (q - \h^\ell w_+ q) = \eta (q - \h^{\ell-2} w_+ q)$
which suffice to show that if we consider elements only up to 
nonzero complex multiples,
the sets
\[
	\{wp_-, wp_+ : w \in D_{2k}\}
	\qquad
	\mbox{and}
	\qquad
	\{q - \h^\ell w_+q : \ell \in \Z\}
		= 
	\{q - \h^\ell w_-q : \ell \in \Z\}
\]
are equal.
Since $j$ and $k$ are relatively prime by \Cref{thm:trichotomy}.(iii),
the elements $q - \h^\ell w_+q$ for $0 \leq \ell < k$
are all distinct elements,
none a scalar multiple of any other since their $E_j$-components are equal,
and prime by \Cref{thm:prime}.
Since each divides $x$, their product $\defm P$ also divides $x$.
Note that 
\[
	P = 
	\prod_{\ell = 0}^{k-1} (q - \h^\ell w_\pm q) =
	q^k - w_\pm q^k.
\]
But the right-hand side 
is a $-1$-eigenvector of $\wpm$, 
so $\dsp\smash{\wpm \frac xP = \frac{-x}{-P} = \frac xP}$ lies in $\HKm \inter \HKp$.
\epf


%
%

We now illustrate \Cref{thm:Heven} with some examples.

\bex There is a cohomogeneity one action of 
$G = \SU(3)$ on the sphere $S^7$ with 
isotropy groups 
$K^-={\rm S}\big(\U(2)\times \U(1)\big)$, 
$K^+={\rm S}\big(\U(1)\times \U(2)\big)$, and
$H={\rm S}\big(\U(1)\times \U(1)\times \U(1)\big)$%
~\cite[Table E]{grovewilkingziller2008}. 
Observe that $\rk G = \rk H$ and we have $n_-=n_+=1$.
Using \Cref{rmk:riemann} and the table in Grove et al.~\cite{grovewilkingziller2008}, 
one deduces that $k=3$.
Here $H^*_{H}$ is just the quotient ring 
$\Q[t_1, t_2, t_3]/(t_1+t_2+t_3)$, which admits a natural $\Sigma_3$-action
permuting the generators $t_j$.
Within $\HH$ we have $\HKm = (\HH)^{\ang{(1 \ 2)}}$
and $\HKp = (\HH)^{\ang{(2 \ 3)}}$,
whose intersection is $(\HH)^{\Sigma_3} \iso \H_{\SU(3)}$.
\Cref{thm:Heven} implies 
\[
	H^*_{\SU(3)} S^7 = 
	H^*_{\SU(3)} \ox H^* S^7 .
\]
This is in fact a known result, being equivalent to 
the equivariant formality of the action~\cite[Cor.~1.3]{goertschesmare2014}
since there is only one possible 
graded $H^*_{\SU(3)}$-algebra structure for 
a free graded $H^*_{\SU(3)}$-module 
on generators of degrees $0$ and $7$.
\eex

Similar calculations can be made for any of the last five examples 
in Grove et al.~\cite[Table E]{grovewilkingziller2008}. 
Note that they are all equivariantly formal actions. 
This is not the case in the next example.

\bex\label{EG:PuettmannWeyl} 
The left action of $\SU(3)$ on itself given by
$(A,B)		\lmt 	ABA^\top$
has cohomogeneity one~\cite[Example 5.5]{puettmann2009}. 
One can see that relative to the canonical metric on $\SU(3)$, 
there is a
transversal geodesic segment joining the two singular orbits, along which the
isotropies are 
$K^+=\SO(3)$,
$K^-=\SU(2) \x \{1\}$, 
and
$H=\SO(2) \x \{1\}$.
The Weyl group $W(\g)$ turns out to be $D_4$, 
but the induced symmetry group $\Xi/W(H)$ is 
just $\Aut \SO(2) = D_2$, so we are in the case $k = 1$.  
If we write $x$ for a generator of $H^*_{\SO(2)}=\Q[x]$, 
then the canonical maps send $H^*_{\SO(3)}$ and $H^*_{\SU(2)}$ both  
isomorphically to $\Q[x^2]$. 
Since $n_- = n_+ = 1$, one concludes
\[
	H^*_{\SU(3)}\SU(3) \iso \Q[x^2] \ox H^* S^3.
\]
\eex

\brmk\label{rmk:CMKeven}
 Again, as in \Cref{rmk:CMKodd}, 
 one can deduce directly that for $K^\pm$ and $H$ as in \Cref{thm:Heven}, 
 $H^*_G M$ is Cohen--Macaulay as an $\HG$-module.
 This time, one notices that the equivariant cohomology is
 the direct sum of two copies of $H^*_{K^+}\inter H^*_{K^-}=(H^*_H)^{D_{2k}}$,
 the latter being a $H^*_G$-algebra which is 
 finitely generated as an $\HG$-module and Cohen--Macaulay as a ring. 
 \ermk

\brmk
When $\rk G$ is
also equal to $\rk H$, the $G$-action is equivariantly formal. 
In this case $M$ is odd-dimensional and the 
situation is particularly nice because the 
restricted action 
of a maximal torus of $G$ is of \emph{GKM type}
in a sense recently defined
by the third author~\cite[Example 4.17]{he2016localization}.
\ermk


\nd\footnotesize{%
\nd\textsc{Department of Mathematics, 
		University of Toronto}\\ 
	\url{jcarlson@math.toronto.edu}
	
\medskip	

\nd\textsc{Fachbereich Mathematik und Informatik, 
		Philipps-Universit\"at Marburg}\\
		\url{goertsch@mathematik.uni-marburg.de}

\medskip

\nd\textsc{Yau Mathematical Sciences Center, 
			Tsinghua University}\\ 
		\url{che@math.tsinghua.edu.cn}

\medskip

\nd\textsc{Department of Mathematics and Statistics, 
		University of Regina}\\
	\url{mareal@math.uregina.ca}
}

\end{document}